\documentclass{article}%
\usepackage{amsmath}
\usepackage{amsfonts}
\usepackage{amssymb}
\usepackage{graphicx}
\usepackage{hyperref}%
\newtheorem{theorem}{Theorem}

\newtheorem{definition}[theorem]{Definition}
\newtheorem{example}[theorem]{Example}

\newtheorem{proposition}[theorem]{Proposition}
\newtheorem{remark}[theorem]{Remark}

\newenvironment{proof}[1][Proof]{\noindent\textbf{#1.} }{\ \rule{0.5em}{0.5em}}
\begin{document}

\title{Outer invariance entropy for discrete-time linear systems on Lie groups}
\author{Fritz Colonius\\Institut f\"{u}r Mathematik, Universit\"{a}t Augsburg, Augsburg, Germany
\and Jo\~{a}o A. N. Cossich\thanks{Partly supported by Proyecto FONDECYT No.
3200819} and Alexandre J. Santana\\Departamento de Matem\'{a}tica, Universidade Estadual de Maring\'{a}\\Maring\'{a}, Brazil}
\maketitle

\begin{abstract}
We introduce discrete-time linear control systems on connected Lie groups and
present an upper bound for the outer invariance entropy of admissible pairs
$(K,Q)$. If the stable subgroup of the uncontrolled system is closed and $K$
has positive measure for a left invariant Haar measure, the upper bound
coincides with the outer invariance entropy.

\end{abstract}


\textbf{MSC\ 2020.} 93B05, 37B40, 94A17, 16W20

\section{Introduction}

In 2009, Colonius and Kawan \cite{ColoniusKawan} introduced the theory of
invariance entropy for control systems. This concept is closely related to
feedback entropy introduced by Nair, Evans, Mareels and Moran \cite{NEMM04} in
an engineering context. Specifically, for invariance entropy a pair $(K,Q)$ of
nonempty subsets of the state space is called admissible if $K$ is a compact
subset of $Q$ and for each $x\in K$ there exists a control $v$ such that the
trajectory $\varphi(\mathbb{R}^{+},x,v)\subset Q$. For $T>0$ denote by
$r_{\mathrm{inv}}(T,K,Q)$ the minimal number of controls $u$ such that for
every initial point $x\in K$ there is $u$ with trajectory $\varphi
([0,T],x,u)\subset Q$. Then the invariance entropy is the exponential growth
rate of these numbers as $T$ tends to infinity,
\[
h_{\mathrm{inv}}(K,Q):={\limsup_{T\rightarrow\infty}}\frac{1}{T}\log
r_{\mathrm{inv}}(T,K,Q).
\]
From this paper the theory was developed culminating in Kawan's book
\cite{Kawa13} that compiled all the theory achieved until 2013. In particular,
this book also developed the theory of invariance entropy for discrete time.
Most of the results lead to more explicit expressions when studied in the case
of linear systems. For linear discrete-time systems on Euclidean spaces the
present authors in \cite{Cocosa4} described the invariance pressure and, in
particular, the invariance entropy for subsets $K$ of the control set $D$,
which is unique and bounded under hyperbolicity assumptions. For the
hyperbolic theory of invariance entropy for continuous-time nonlinear control
systems, cf. Da Silva and Kawan \cite{DaSilK16} and also \cite{DaSilK19}.
Invariance entropy for continuous-time linear systems on Lie groups has been
analyzed by Ayala, Da Silva, Jouan, and Zsigmond \cite{AySilJuZs}, cf. also
the references therein for the theory of continuous-time linear control
systems on Lie groups.

For discrete time, the present paper introduces linear control systems on
connected Lie groups and starts the investigation of invariance entropy. We
define discrete-time linear systems on connected Lie groups $G$ as systems of
the form
\[
\Sigma\ :\ g_{k+1}=f(g_{k},u_{k}),\ \ u_{k}\in U,
\]
where $f_{u}(\cdot):=f(\cdot,u):G\rightarrow G$ is an automorphism for $u=0$
and otherwise a diffeomorphism such that $f_{u}(g)=f_{u}(e)\cdot f_{0}(g)$. We
refer to $g_{k+1}=f(g_{k},0)=f_{0}(g_{k})$ as the uncontrolled system.

Outer invariance entropy weakens the requirements on the trajectories: the
trajectories are allowed to go out to $\varepsilon$-neighborhoods
$N_{\varepsilon}(Q)$ of $Q$, and
\[
h_{\mathrm{inv,out}}(K,Q):=\lim_{\varepsilon\rightarrow0}h_{\mathrm{inv}%
}(K,N_{\varepsilon}(Q)).
\]
Our main result establishes that an upper bound for outer invariance entropy
of an admissible pair $(K,Q)$ of discrete-time linear systems is given by the
sum of the logarithms of the modulus of the eigenvalues $\lambda$ of the
differential $(df_{0})_{e}$ with $|\lambda|>1$, where $e$ is the identity of
$G$. The upper bound coincides with the outer invariance entropy, if the
stable subgroup of the uncontrolled system (cf. the definition after formula
(\ref{g})) is closed and $K$ has positive measure for a left invariant Haar
measure $\mu$ on $G$. This bears some similarity to the characterization of
outer invariance entropy in the continuous-time case, cf. Da Silva
\cite{DaSil14}.

Section \ref{Section2} presents the main concepts for discrete-time linear
control system on Lie groups and some examples. In Section \ref{Section2a} we
introduce discrete-time linear control system induced on homogeneous spaces
and Section \ref{Section3} proves the main result of the paper, the
characterization of outer invariance entropy for discrete-time linear control
systems on Lie groups. Finally, Section \ref{Section5} derives a sufficient
condition for closedness of the stable subgroup of the uncontrolled system.

\section{Discrete-time linear control systems on Lie groups\label{Section2}}

In this section we present our definition of discrete-time linear control
systems on Lie groups in analogy to the definition of the continuous-time
linear control systems on Lie groups, derive some properties, and provide
several examples.

Recall (cf. e.g. Sontag \cite{Son98}) that a discrete-time control system on a
topological space $M$ is given by difference equations
\[
x_{k+1}=f(x_{k},u_{k}),\ \ u_{k}\in U,
\]
where $k\in\mathbb{N}_{0}$, the control range $U$ is a nonempty set and
$f:M\times U\rightarrow M$ is a map such that $f_{u}(\cdot):=f(\cdot
,u):M\rightarrow M$ is continuous for each $u\in U$. For initial value
$x_{0}:=x\in M$ and control $u=(u_{i})_{i\in\mathbb{N}_{0}}\in\mathcal{U}%
=U^{\mathbb{N}_{0}}$ the solution of this system is given by%
\[
\varphi(k,x,u)=\left\{
\begin{array}
[c]{lll}%
x & \mbox{for} & k=0\\
f_{u_{k-1}}\circ\cdots\circ f_{u_{1}}\circ f_{u_{0}}(x) & \mbox{for} & k\geq1
\end{array}
\right.  .
\]
Where convenient, we also write $\varphi_{k,u}:=\varphi(k,\cdot,u)$. Now we
present our definition of discrete-time linear control systems on Lie groups.

\begin{definition}
Let $U\subset\mathbb{R}^{m}$ with $0\in U$. A discrete-time control system
\[
\Sigma\ :\ g_{k+1}=f(g_{k},u_{k}),\ \ u_{k}\in U,
\]
on a connected Lie group $G$ is linear if $f_{0}=f(\cdot,0):G\rightarrow G$ is
an automorphism and for each $u\in U$, $f_{u}:G\rightarrow G$ satisfies
\[
f_{u}(g)=f_{u}(e)\cdot f_{0}(g)=L_{f_{u}(e)}(f_{0}(g))\text{ for all }g\in G.
\]

\end{definition}

Here \textquotedblleft$\ \cdot$\textquotedblright\ denotes the product of $G$,
it will be omitted when it is clear by the context. Thus $f_{u}(g)$ is given
by left translation $L_{f_{u}(e)}$ of $f_{0}(g)$ by $f_{u}(e)$, and
$f_{0}(g)^{-1}=f_{0}(g^{-1})$ implies that for each $u\in U$ the map
$f_{u}:G\rightarrow G$ is a diffeomorphism with inverse%
\[
(f_{u})^{-1}(g)=(f_{0})^{-1}\left(  \left(  f_{u}(e)\right)  ^{-1}\cdot
g\right)  =(f_{0})^{-1}\circ L_{\left(  f_{u}(e)\right)  ^{-1}}(g),g\in G.
\]

\begin{example}
\label{Example2}Consider on the additive Lie group $G=\mathbb{R}^{d}$ the
control system given by
\[
x_{k+1}=Ax_{k}+Bu_{k},\ \ u_{k}\in U,
\]
where $A\in GL(d,\mathbb{R}),B\in\mathbb{R}^{d\times m}$, and $0\in
U\subset\mathbb{R}^{m}$. In this case, $f:\mathbb{R}^{d}\times U\rightarrow
\mathbb{R}^{d}$ is given by $f(x,u)=Ax+Bu$. Note that $f_{0}(x)=Ax$ is an
automorphism of $\mathbb{R}^{d}$ and $f_{u}(e)=f_{u}(0)=Bu$, hence
$f_{u}(x)=f_{0}(x)+f_{u}(0)=f_{u}(0)+f_{0}(x)$. In this case, the solutions
are given by
\[
\varphi(k,x,u)=A^{k}x+\sum_{j=0}^{k-1}A^{k-1-j}Bu_{j}.
\]

\end{example}

The next result shows that for linear systems the solution starting in a point
$g$ is a translation of the solution starting in the identity $e$.

\begin{proposition}
\label{proposition5}Consider a discrete-time linear control system
$g_{k+1}=f(g_{k},u_{k})$, $u_{k}\in U$, on a Lie group $G$. Then it follows
for all $g\in G$ and $u=(u_{i})\in\mathcal{U}$ that
\begin{equation}
\varphi(k,g,u)=\varphi(k,e,u)f_{0}^{k}(g)\text{ for all }k\in\mathbb{N}.
\label{solution}%
\end{equation}

\end{proposition}

\begin{proof}
The proof will follow by induction over $k\in\mathbb{N}$. Note initially that
for $g\in G$ and $u\in\mathcal{U}$ we have
\[
\varphi(1,g,u)=f_{u_{0}}(g)=f_{u_{0}}(e)f_{0}(g)=\varphi(1,e,u)f_{0}(g).
\]
Now, suppose that the equality holds for $k\in\mathbb{N}$, that is,
$\varphi(k,g,u)=\varphi(k,e,u)f_{0}^{k}(g)$. With the $k$-shift to the right
given by $\Theta_{k}(u_{i})_{i\in\mathbb{N}}=(u_{i+k})_{i\in\mathbb{N}}$, this
implies
\begin{align*}
\varphi(k+1,g,u)  &  =\varphi(1,\varphi(k,g,u),\Theta_{k}u)=\varphi
(1,\varphi(k,e,u)f_{0}^{k}(g),\Theta_{k}u)\\
&  =f_{u_{k}}(\varphi(k,e,u)f_{0}^{k}(g))=f_{u_{k}}(e)f_{0}(\varphi
(k,e,u)f_{0}^{k}(g))\\
&  =f_{u_{k}}(e)f_{0}(\varphi(k,e,u))f_{0}^{k+1}(g)=f_{u_{k}}(\varphi
(k,e,u))f_{0}^{k+1}(g)\\
&  =\varphi(k+1,e,u)f_{0}^{k+1}(g).
\end{align*}

\end{proof}

\begin{remark}
For a continuous-time linear control system on a connected Lie group the
solution is given by%
\[
\varphi(t,g,u)=\varphi(t,e,u)\varphi(t,g,0),
\]
cf. Ayala, Da Silva, and Zsigmond \cite[formula (7)]{AyalDSZ17}. The solution
formula (\ref{solution}) for discrete-time linear systems is analogous.
\end{remark}

\begin{example}
\label{example5} Consider the connected component of the identity for the
affine group
\[
\emph{Aff}(2,\mathbb{R})_{0}=\left\{  \left[
\begin{array}
[c]{cc}%
x & y\\
0 & 1
\end{array}
\right]  ;\ \ x>0\ \mbox{ and }\ y\in\mathbb{R}\right\}  .
\]
This group can be seen as $\mathbb{R}_{+}\times\mathbb{R}$ provided with the
product
\[
(x_{1},y_{1})\cdot(x_{2},y_{2})=(x_{1}x_{2},x_{1}y_{2}+y_{1}).
\]
Define $f:\emph{Aff}(2,\mathbb{R})_{0}\times U\rightarrow\emph{Aff}%
(2,\mathbb{R})_{0}$ as $f((x,y),u)=(xe^{u},ye^{2+u}+u)$, where $U\subset
\mathbb{R}$ with $0\in U$. Elementary calculations show that $f_{0}%
:\emph{Aff}(2,\mathbb{R})_{0}\rightarrow\emph{Aff}(2,\mathbb{R})_{0}$ is an
automorphism and $f_{u}(x,y)=f_{u}(1,0)\cdot f_{0}(x,y)$, for all
$(x,y)\in\emph{Aff}(2,\mathbb{R})_{0}$. Hence,
\[
(x_{k+1},y_{k+1})=f((x_{k},y_{k}),u_{k}),\ u_{k}\in U,
\]
is a discrete-time linear control system on $\emph{Aff}(2,\mathbb{R})_{0}$.

For $u=(u_{i})_{i\in\mathbb{N}_{0}}\in\mathcal{U}$, denote by $S_{k}(u)$ the
sum $\sum_{i=0}^{k}u_{i}$. Proposition \ref{proposition5} and induction over
$k$ imply that for all $k\in\mathbb{N}_{0}$, $(x,y)\in\emph{Aff}%
(2,\mathbb{R})_{0}$ and $u\in\mathcal{U}$,%
\begin{align*}
\varphi(k,(x,y),u)  &  =\varphi(k,(1,0),u)f_{0}^{k}(x,y)\\
&  =\left(  e^{S_{k-1}(u)},\sum_{j=0}^{k-1}u_{j}e^{2(k-1-j)+S_{k-1}%
(u)-S_{j}(u)}\right)  \left(  x,ye^{2k}\right) \\
&  =\left(  xe^{S_{k-1}(u)},ye^{2k+S_{k-1}(u)}+\sum_{j=0}^{k-1}u_{j}%
e^{2(k-1-j)+S_{k-1}(u)-S_{j}(u)}\right)  .
\end{align*}

\end{example}

\begin{example}
\label{Ex4} Let
\[
G=\left\{  \left[
\begin{array}
[c]{ccc}%
1 & x_{2} & x_{1}\\
0 & 1 & x_{3}\\
0 & 0 & 1
\end{array}
\right]  ;\ (x_{1},x_{2},x_{3})\in\mathbb{R}^{3}\right\}
\]
be the Heisenberg group. Note that $G$ is diffeomorphic to $\mathbb{R}^{3}$
with the product
\[
(x_{1},x_{2},x_{3})\cdot(y_{1},y_{2},y_{3})=(x_{1}+y_{1}+x_{2}y_{3}%
,x_{2}+y_{2},x_{3}+y_{3}).
\]
Consider the automorphism $f_{0}:G\rightarrow G$ given by
\[
f_{0}(x_{1},x_{2},x_{3})=\left(  x_{1}+x_{2}+\dfrac{x_{2}^{2}}{2},x_{2}%
,x_{2}+x_{3}\right)
\]
and for each $u\in U$, define the diffeomorphism $f_{u}:G\rightarrow G$ by
\[
f_{u}(x_{1},x_{2},x_{3})=\left(  x_{1}+x_{2}+\dfrac{x_{2}^{2}}{2}%
+ux_{2}+ux_{3}-\dfrac{u}{2}-\dfrac{u^{2}}{3},x_{2}+u,x_{2}+x_{3}-\dfrac{u}%
{2}\right)  .
\]
It is not difficult to see that for each $g\in G$, $f_{u}(g)=f_{u}(e)f_{0}%
(g)$. Hence, the system
\[
\Sigma\ :\ g_{k+1}=f(g_{k},u_{k})
\]
on $G$ is linear. By Proposition \ref{proposition5}, the solutions are given
by
\[
\varphi(k,g,u)=\varphi(k,e,u)f_{0}^{k}(g)=\varphi(k,e,u)\left(  x_{1}%
+kx_{2}+\dfrac{k}{2}x_{2}^{2},x_{2},kx_{2}+x_{3}\right)  .
\]

\end{example}

\begin{remark}
\label{Obs5} For a discrete-time linear control system $g_{k+1}=f(g_{k}%
,u_{k})$, $u_{k}\in U$, on a Lie group $G$, we have that $(df_{0}%
)_{e}:\mathfrak{g}\rightarrow\mathfrak{g}$ is a Lie algebra isomorphism,
because $f_{0}$ is an automorphism of $G$. Also, we can see that $(df_{0}%
^{n})_{e}=[(df_{0})_{e}]^{n}$ for each $n\in\mathbb{Z}$ and hence for each
$X\in\mathfrak{g}$%
\[
f_{0}^{n}(\exp X)=\exp([(df_{0})_{e}]^{n}X).
\]

\end{remark}

Given the automorphism $f_{0}$, we can define several Lie subalgebras that are
intrinsically associated with its dynamics. Consider an eigenvalue $\alpha$ of
$(df_{0})_{e}$ and its generalized eigenspace
\[
\mathfrak{g}_{\alpha}=\{X\in\mathfrak{g};\ ((df_{0})_{e}-\alpha)^{n}X=0,\text{
for some }n\geq1\}.
\]

The following proposition, whose proof can be found in Ayala and Da Silva
\cite[Proposition 2.1]{AyaDS16}, will be useful in order to define dynamical
subalgebras of $\mathfrak{g}$ related to the differential of $f_{0}$ at $e\in
G$.

\begin{proposition}
\label{dynliealg} If $\alpha$ and $\beta$ are eigenvalues of $(df_{0})_{e}$,
then
\[
[\mathfrak{g}_{\alpha},\mathfrak{g}_{\beta}]\subset\mathfrak{g}_{\alpha\beta
},
\]
where $\mathfrak{g}_{\alpha\beta}=\{0\}$ if $\alpha\beta$ is not an eigenvalue
of $(df_{0})_{e}$.
\end{proposition}

Due to Proposition \ref{dynliealg}, we define the unstable, center, and
unstable Lie subalgebras (cf. \cite{AyaDS16}) by%
\begin{equation}
\mathfrak{g}^{+}=\bigoplus_{|\alpha|>1}\mathfrak{g}_{\alpha},\ \ \mathfrak{g}%
^{0}=\bigoplus_{|\alpha|=1}\mathfrak{g}_{\alpha},\ \ \mathfrak{g}%
^{-}=\bigoplus_{0<|\alpha|<1}\mathfrak{g}_{\alpha}. \label{g}%
\end{equation}

Proposition \ref{dynliealg} implies that $\mathfrak{g}^{+}$ and $\mathfrak{g}%
^{-}$ are nilpotent Lie subalgebras of $\mathfrak{g}$. Since $(df_{0})_{e}$ is
a Lie algebra isomorphism, the decomposition $\mathfrak{g}=\mathfrak{g}%
^{+}\oplus\mathfrak{g}^{0}\oplus\mathfrak{g}^{-}$ holds. The center-unstable
and center-stable Lie subalgebras are%
\[
\mathfrak{g}^{+,0}=\mathfrak{g}^{+}\oplus\mathfrak{g}^{0}\ \text{ and
}\ \mathfrak{g}^{-,0}=\mathfrak{g}^{-}\oplus\mathfrak{g}^{0}.
\]
We will also need the connected Lie groups $G^{-,0}$, $G^{+,0}$, $G^{-}$ and
$G^{+}$ corresponding to $\mathfrak{g}^{-,0}$, $\mathfrak{g}^{+,0}$,
$\mathfrak{g}^{-}$ and $\mathfrak{g}^{+}$, respectively. In particular, we
refer to $G^{-}$ and $G^{+}$ as the stable and the unstable subgroup of the
uncontrolled system, resp.

\begin{remark}
\label{Obs12} The restrictions of $(df_{0})_{e}$ to the Lie subalgebras
$\mathfrak{g}^{+}$, $\mathfrak{g}^{0}$, and $\mathfrak{g}^{-}$ satisfy:
\[
|[(df_{0})_{e}]^{n}X|\geq c\sigma^{-n}|X|,\ \text{ for any }\ X\in
\mathfrak{g}^{+},\ n\in\mathbb{N},
\]
and
\[
|[(df_{0})_{e}]^{n}Y|\leq c^{-1}\sigma^{n}|Y|,\ \text{ for any }
\ Y\in\mathfrak{g}^{-},\ n\in\mathbb{N},
\]
for some $c\geq1$ and $\sigma\in(0,1)$ and, for all $a>0$ and $Z\in
\mathfrak{g}^{0}$ it holds that
\[
|[(df_{0})_{e}]^{n}Z|\sigma^{a|n|}\rightarrow0\ \text{ as }\ n\rightarrow
\pm\infty.
\]

\end{remark}

\section{Linear systems induced on homogeneous spaces\label{Section2a}}

In this section, we define a induced (discrete-time) linear system on a
homogeneous space $G/H$ from a discrete-time linear system on a Lie group $G$.
This construction will be important to get a formula for the outer invariance
entropy presented in the Theorem \ref{Theo11}.

Consider the discrete-time linear control system
\begin{equation}
g_{k+1}=f(g_{k},u_{k}),\ \ \ u=(u_{i})_{i\in\mathbb{N}_{0}}\in\mathcal{U},
\label{lineardisc}%
\end{equation}
on a Lie group $G$ and let $H$ a Lie subgroup of $G$ which is $f_{0}%
$-invariant. If $H$ is closed, this induces a control system on the
homogeneous space $G/H$ in the following way: define $\bar{f}:G/H\times
U\rightarrow G/H$ as $\bar{f}(gH,u)=f(g,u)H$. The $f_{0}$-invariance of $H$
implies that $\bar{f}$ is well defined. Note also that for each $u\in U$, the
map $\bar{f}_{u}:G/H\rightarrow G/H$ is a diffeomorphism with inverse $\bar
{f}_{u}^{-1}(gH)=f_{u}^{-1}(g)H$, because $\bar{f}_{u}\circ\pi=\pi\circ f_{u}$
and $\bar{f}_{u}^{-1}\circ\pi=\pi\circ f_{u}^{-1}$ are differentiable maps,
where $\pi:G\rightarrow G/H$ is the natural projection. This induces a
discrete-time control system
\begin{equation}
x_{k+1}=\bar{f}(x_{k},u_{k}),\ x_{k}\in G/H,\ u=(u_{i})_{i\in\mathbb{N}_{0}%
}\in\mathcal{U}, \label{induced}%
\end{equation}
on $G/H$ with solutions denoted by $\bar{\varphi}$. Since for each
$k\in\mathbb{N}$, $\bar{f}_{0}^{k}\circ\pi=\pi\circ f_{0}^{k}$, we have that
\[
\pi(\varphi(k,g,u))=\bar{\varphi}(k,\pi(g),u),
\]
for all $k\in\mathbb{N}_{0}$, $g\in G$ and $u\in\mathcal{U}$. In other words,
$(\pi,\mathrm{id}_{\mathcal{U}})$ is a semi-conjugacy between the systems
(\ref{lineardisc}) and (\ref{induced}) (cf. Kawan \cite[Definition
2.4]{Kawa13}). For each $g\in G$ we denote by $\mathcal{L}_{g}$ the left
translation $G/H\ni xH\mapsto gxH\in G/H$. Then we find that
\begin{align}
\bar{\varphi}(k,gH,u)  &  =\pi(\varphi(k,g,u))=\pi(\varphi_{k,u}(e)f_{0}%
^{k}(g))\nonumber\\
&  =\varphi_{k,u}(e)f_{0}^{k}(g)H=\varphi_{k,u}(e)\bar{f}_{0}^{k}%
(gH)\label{linear2}\\
&  =\mathcal{L}_{\varphi_{k,u}(e)}(\bar{f}_{0}^{k}(gH)).\nonumber
\end{align}
This shows that the solutions of (\ref{induced}) satisfy properties similar to
equality (\ref{solution}) for a linear system on $G$, hence we call it the
\textbf{induced linear system }on the homogeneous space $G/H$.

\begin{proposition}
\label{decomp} Let $G$ a Lie group with Lie algebra $\mathfrak{g}$. Assume
that $\mathfrak{g}$ decomposes as $\mathfrak{g}=\mathfrak{h}\oplus
\mathfrak{l}$, where $\mathfrak{h}$ and $\mathfrak{l}$ are $(df_{0})_{e}%
$-invariant Lie subalgebras of $\mathfrak{g}$. Consider the connected Lie
subgroup $H$ of $G$ with Lie algebra $\mathfrak{h}$. If $H$ is closed, then
$(d\bar{f}_{0}^{k})_{eH}=(df_{0})_{e}^{k}|_{\mathfrak{l}}$, for all
$k\in\mathbb{N}_{0}$.
\end{proposition}

\begin{proof}
The $(df_{0})_{e}$-invariance of $\mathfrak{h}$ and $\mathfrak{l}$ allows us
to consider the well defined linear isomorphism $\overline{(df_{0})_{e}%
}:\mathfrak{g}/\mathfrak{h}\rightarrow\mathfrak{g}/\mathfrak{h}$ given by
$\overline{(df_{0})_{e}}(X+\mathfrak{h})=(df_{0})_{e}X+\mathfrak{h}$. If $\pi$
is the natural projection of $G$ onto $G/H$, we can see that $\overline
{(df_{0})_{e}}$ satisfies $\overline{(df_{0})_{e}}\circ(d\pi)_{e}=(d\pi
)_{e}\circ(df_{0})_{e}$. Hence, for all $X\in\mathfrak{g}$ and $k\in
\mathbb{N}_{0}$ we have
\begin{align*}
(d\bar{f}_{0}^{k})_{eH}(X+\mathfrak{h})  &  =d(\bar{f}_{0}^{k}\circ\pi
)_{e}X=d(\pi\circ f_{0}^{k})_{e}X=(d\pi)_{e}\circ(df_{0}^{k})_{e}X\\
&  =(d\pi)_{e}\circ(df_{0})_{e}^{k}X=(\overline{(df_{0})_{e}})^{k}\circ
(d\pi)_{e}X\\
&  =(\overline{(df_{0})_{e}})^{k}(X+\mathfrak{h}).
\end{align*}
By invariance, we can identify $\mathfrak{g}/\mathfrak{h}$ with $\mathfrak{l}$
and, therefore, $\overline{(df_{0})_{e}}$ with $(df_{0})_{e}^{k}%
|_{\mathfrak{l}}$ and the desired equality holds.
\end{proof}

\begin{definition}
A measure $\mu$ on a homogeneous space $G/H$ is a \textbf{$G$-invariant} Borel
measure if $\mu\left(  \mathcal{L}_{g}(A)\right)  =\mu(A)$, for all $g\in G$
and all Borel set $A\subset G/H$.
\end{definition}

The following proposition is an adaptation of Da Silva \cite[Proposition
4.6]{DaSil14} in our context of discrete-time linear systems.

\begin{proposition}
\label{borel}If the connected subgroup $G^{-}$ corresponding to $\mathfrak{g}%
^{-}$ defined in (\ref{g}) is closed, the homogeneous space $G/G^{-}$ admits a
unique (up to a scalar) $G$-invariant Borel measure.
\end{proposition}

\begin{proof}
Denote by $\Delta_{G}$ and $\Delta_{G^{-}}$ the modular functions of $G$ and
$G^{-}$, respectively, cf. Knapp \cite[Chapter VIII, Section 2]{Knapp}. The
result will follow from \cite[Theorem 8.36]{Knapp}, if we can show that
$(\Delta_{G})|_{G^{-}}=\Delta_{G^{-}}$.

Since $G^{-}$ is nilpotent, $\Delta_{G^{-}}(h)=1$ for all $h\in G^{-}$. On the
other hand, for any two eigenvalues $\alpha,\beta$ of $(df_{0})_{e}$ we have
\[
\mathrm{ad}(\mathfrak{g}_{\alpha})^{n}\mathfrak{g}_{\beta}\subset
\mathfrak{g}_{\alpha^{n}\beta},\ \ \forall\ n\in\mathbb{N},
\]
which implies that for each $X\in\mathfrak{g}^{-}$, $\text{ad}(X):\mathfrak{g}%
\rightarrow\mathfrak{g}$ is a nilpotent linear map. Writing for $X\in
\mathfrak{g}^{-}$
\[
X=\sum_{|\alpha|<1}a_{\alpha}X_{\alpha},\quad a_{\alpha}\in\mathbb{R},
\]
where the sum is taken over all eigenvalues $\alpha$ of $(df_{0})_{e}$ with
$|\alpha|<1$, we get
\[
\mathrm{tr}(\mathrm{ad}(X))=\sum_{|\alpha|<1}a_{\alpha}\mathrm{tr}%
(\mathrm{ad}(X_{\alpha}))=0.
\]
Now, consider $g\in G^{-}$. Then $g$ can be written as
\[
g=\exp X_{1}\cdots\exp X_{r},
\]
for some $X_{1},\ldots,X_{r}\in\mathfrak{g}^{-}$, because $G^{-}$ is
connected. Therefore,
\begin{align*}
\Delta_{G}(g)  &  =\Delta_{G}(\exp X_{1}\cdots\exp X_{r})=\Delta_{G}(\exp
X_{1})\cdots\Delta_{G}(\exp X_{r})\\
&  =|\det\left(  \text{\textrm{Ad}}(\exp X_{1})\right)  |\cdots|\det\left(
\mathrm{Ad}(\exp X_{r})\right)  |\\
&  =\left\vert \det\left(  e^{\mathrm{ad}(X_{1})}\right)  \right\vert
\cdots\left\vert \det\left(  e^{\mathrm{ad}(X_{r})}\right)  \right\vert \\
&  =\left\vert e^{\mathrm{tr}(\mathrm{ad}(X_{1}))}\right\vert \cdots\left\vert
e^{\mathrm{tr}(\mathrm{ad}(X_{r}))}\right\vert =1.
\end{align*}
This shows that $(\Delta_{G})|_{G^{-}}=\Delta_{G^{-}}$ and, by \cite[Theorem
8.36]{Knapp}, the result follows.
\end{proof}

\begin{remark}
\label{remarklast} According to \cite[Theorem 8.36]{Knapp}, if $\nu_{G}$ and
$\nu_{H}$ are the left invariant Haar measures on $G$ and $H$, respectively,
the $G$-invariant Borel measure $\mu$ on $G/H$ can be normalized so that
\[
\int_{G}\phi(g)\ d\nu_{G}(g)=\int_{G/H}\int_{H}\phi(gh)\ d\nu_{H}%
(h)\ d\mu(gH),
\]
for all continuous function $\phi:G\rightarrow\mathbb{R}$ with compact
support. Note that if $K\subset G$ is compact, then $\nu_{G}(K)>0$ implies
$\mu(\pi(K))>0$. This follows from the fact that $\overline{\chi_{K}}%
(gH)=\int_{H}\chi_{K}(gh)\ d\nu_{H}(h)$ is a bounded positive function and
$\overline{\chi_{K}}>0$ if and only if $\chi_{\pi(K)}>0$.
\end{remark}

\section{Outer invariance entropy\label{Section3}}

In this section we prove our main result that provides an upper bound for
outer invariance entropy of discrete-time linear control systems on Lie
groups. We begin by recalling the definition of invariance entropy and outer
invariance entropy as given in Kawan \cite[Definitions 2.2 and 2.3]{Kawa13}.

Consider a discrete-time control system
\[
\Sigma\ :\ x_{k+1}=f(x_{k},u_{k}),\ \ u_{k}\in U,\ x\in M,
\]
with solutions $\varphi(k,x,u),k\in\mathbb{N}_{0}$. A pair $(K,Q)$ of nonempty
subsets of $M$ is called \textbf{\emph{admissible}}, if $K$ is compact and for
each $x\in K$, there exists $u\in\mathcal{U}$ such that $\varphi(k,x,u)\in Q$
for all $k\in\mathbb{N}_{0}$.

Given an admissible pair $(K,Q)$ and $n\in\mathbb{N}$, we say that a set
$\mathcal{S}\subset\mathcal{U}$ is $(n,K,Q)$-spanning if
\[
\forall\ x\in K\ \exists\ u\in\mathcal{S~}\forall\ j\in\{1,\ldots
,n\}:\ \varphi(j,x,u)\in Q.
\]
Denote by $r_{\mathrm{inv}}(n,K,Q)$ the minimal number of elements such a set
can have (if there is no finite set with this property we set $r_{\mathrm{inv}%
}(n,K,Q)=\infty$).

The existence of $(n,K,Q)$-spanning sets is guaranteed, since $\mathcal{U}$ is
$(n,K,Q)$-spanning for every $n\in\mathbb{N}$.

\begin{definition}
Given an admissible pair $(K,Q)$ for a discrete-time control system $\Sigma$,
the invariance entropy of $(K,Q)$ is defined by
\[
h_{\mathrm{inv}}(K,Q)=h_{\mathrm{inv}}(K,Q;\Sigma):=\limsup_{n\rightarrow
\infty}\frac{1}{n}\log r_{\mathrm{inv}}(n,K,Q).
\]

\end{definition}

Here, and throughout the paper, $\log$ denotes the logarithm with base $2$.

The invariance entropy of $(K,Q)$ measures the exponential growth rate of the
minimal number of control functions sufficient to stay in $Q$ when starting in
$K$ as time tends to infinity. Hence, invariance entropy is a nonnegative
(possibly infinite) quantity which is assigned to an admissible pair $(K,Q)$.
For our main result we need the following related quantity. Note that for an
admissible pair $(K,Q)$ also every pair $(K,N_{\varepsilon}(Q)),\varepsilon
>0$, is admissible, where $N_{\varepsilon}(Q)=\{x\in M\left\vert
d(x,Q)<\varepsilon\right.  \}$ denotes the $\varepsilon$-neighborhood of $Q$.

\begin{definition}
Given an admissible pair $(K,Q)$ such that $Q$ is closed in $M$ and a metric
$d$ on $M$, we define the outer invariance entropy of $(K,Q)$ by
\[
h_{\mathrm{inv,out}}(K,Q):=h_{\mathrm{inv,out}}(K,Q,\Sigma,d):=\lim
_{\varepsilon\rightarrow0}h_{\mathrm{inv}}(K,N_{\varepsilon}(Q))=\sup
_{\varepsilon>0}h_{\mathrm{inv}}(K,N_{\varepsilon}(Q)).
\]

\end{definition}

Obviously, the inequality $h_{\mathrm{inv}}(K,Q)\geq h_{\mathrm{inv,out}%
}(K,Q)$ holds. The next proposition (cf. Kawan \cite[Proposition 2.13]%
{Kawa13}) describes the behavior of outer invariance entropy under
semi-conjugacy. Recall that a semi-conjugacy of two discrete-time control
systems on metric spaces $M_{1}$ and $M_{2}$, respectively, given by%
\[
x_{k+1}=f_{1}(x_{k},u_{k}),u_{k}\in U\text{, and }\ y_{k+1}=f_{2}(y_{k}%
,v_{k}),V_{k}\in V,
\]
with solutions $\varphi_{1}$ and $\varphi_{2}$, resp., is a pair $(\pi,h)$ of
continuous maps $\pi:M_{1}\rightarrow M_{2}$ and $h:U^{\mathbb{N}_{0}%
}\rightarrow V^{\mathbb{N}_{0}}$ such that%
\begin{equation}
\pi(\varphi_{1}(k,x,u))=\varphi_{2}(k,\pi(x),h(u))\text{ for all }%
k\in\mathbb{N}_{0},\ x\in M,u\in U^{\mathbb{N}_{0}}. \label{conj2}%
\end{equation}

\begin{proposition}
\label{proposition9}Consider two discrete-time control systems $\Sigma_{1}$
and $\Sigma_{2}$ and let $(\pi,h)$ be a semi-conjugacy from $\Sigma_{1}$ to
$\Sigma_{2}$. Then every admissible pair $(K,Q)$ for $\Sigma_{1}$ with compact
$Q$ defines an admissible pair $(\pi(K),\pi(Q))$ for $\Sigma_{2}$ and
\[
h_{\mathrm{inv,out}}(K,Q;\Sigma_{1})\geq h_{\mathrm{inv,out}}(\pi
(K),\pi(Q);\Sigma_{2}).
\]

\end{proposition}

Now, we will briefly recall the notion of topological entropy as defined by
Bowen and Dinaburg (cf. e.g. Walters \cite{Walt82}). Let $(M,d)$ be a metric
space and $\phi:M\rightarrow M$ be a continuous map. Given a compact set
$K\subset X$ and $n\in\mathbb{N}$ we say that a set $F\subset M$ is
$(n,\varepsilon)$-spanning set for $K$ with respect to $\phi$ if, for every
$y\in K$, there exists $x\in F$ such that
\[
d(\phi^{j}(x),\phi^{j}(y))<\varepsilon,\ \text{ for all }\ j\in\{0,1,\ldots
,n\}.
\]
If we denote by $r_{n}(\varepsilon,K)$ the minimal cardinality of an
$(n,\varepsilon)$-spanning set for $K$ with respect to $\phi$, the topological
entropy of $\phi$ over $K$ is defined by
\[
h_{\mathrm{top}}(\phi,K)=\lim_{\varepsilon\rightarrow0}\limsup_{n\rightarrow
\infty}\frac{1}{n}\log r_{n}(\varepsilon,K)
\]
and the topological entropy of $\phi$ is
\[
h_{\mathrm{top}}(\phi)=\sup_{K\text{ compact}}h_{\mathrm{top}}(\phi,K).
\]
Alternatively, topological entropy can be defined using $(n,\varepsilon
)$-separated sets $E\subset K$ with respect to $\phi$ which are defined as
sets satisfying
\[
\forall\ x,y\in E:x\neq y\Rightarrow\exists j\in\{0,\ldots,n\}:\ d(\phi
^{j}(x),\phi^{j}(y))>\varepsilon.
\]
Let $s_{n}(\varepsilon,K)$ denote the largest cardinality of any
$(n,\varepsilon)$-separated subset of $K$ with respect to $\phi$. If $E$ is an
$(n,\varepsilon)$-separated subset of $K$ of maximal cardinality, then $E$ is
$(n,\varepsilon)$-spanning for $K$. The topological entropy of $\phi$ over $K$
can also be characterized as the limit
\[
h_{\mathrm{top}}(\phi,K)=\lim_{\varepsilon\rightarrow0}\limsup_{n\rightarrow
\infty}\frac{1}{n}\log s_{n}(\varepsilon,K).
\]
Bowen \cite[Corollary 16]{Bow} shows that for$\text{ an endomorphism }%
\phi\text{ of a }$Lie group $G$ the topological entropy is
\begin{equation}
h_{\mathrm{top}}(\phi)=\sum_{|\lambda|>1}\log|\lambda|, \label{Bow1}%
\end{equation}
where the sum is taken over all eigenvalues $\lambda$ of $d(\phi)_{e}$ with
$|\lambda|>1$.

Now we formulate the main result of this paper.

\begin{theorem}
\label{Theo11} Let $(K,Q)$ be an admissible pair for the discrete-time linear
system $\Sigma:g_{k+1}=f(g_{k},u_{k})$, $u_{k}\in U$, on a connected Lie group
$G$. Assume that $Q$ is compact.

\begin{itemize}
\item[i)] Then the outer invariance entropy satisfies
\[
h_{\mathrm{inv,out}}(K,Q)\leq\sum_{|\lambda|>1}\log|\lambda|=h_{\mathrm{top}%
}(f_{0}),
\]
where the sum is taken over all eigenvalues $\lambda$ of $(df_{0})_{e}$ with
$|\lambda|>1$.

\item[ii)] Moreover, suppose that the connected subgroup $G^{-}$ corresponding
to $\mathfrak{g}^{-}$ defined in (\ref{g}) is closed and that $\nu_{G}(K)>0$
for a left invariant Haar measure $\nu_{G}$ on $G$. Then
\[
h_{\mathrm{inv,out}}(K,Q)=\sum_{\lambda}\log|\lambda|=h_{\mathrm{top}}%
(f_{0}),
\]
where the sum is taken over all eigenvalues $\lambda$ of $(df_{0})_{e}$ with
$|\lambda|>1$.
\end{itemize}
\end{theorem}

\begin{proof}
(i) The characterization of the topological entropy of the automorphism
$f_{0}$ follows from Bowen's result (\ref{Bow1}). Since $Q$ is compact, Kawan
\cite[Proposition 2.5]{Kawa13}\ implies that the invariance entropy does not
depend on the metric. Hence we may choose a left invariant Riemannian metric
$\varrho$ on $G$. This yields for $g\in G$,
\[
\varrho(\varphi_{k,u}(g),\varphi_{k,u}(h))=\varrho(\varphi_{k,u}(e)f_{0}%
^{k}(g),\varphi_{k,u}(e)f_{0}^{k}(h))=\varrho(f_{0}^{k}(g),f_{0}^{k}(h)).
\]
Now, fix $\varepsilon>0$ and $n\in\mathbb{N}$. Let $E\subset K$ be an
$(n,\varepsilon)$-separated subset of $K$ with respect to $f_{0}$ of maximal
cardinality $s_{n}(\varepsilon,K)$. Since $(K,Q)$ is admissible, for each
$h\in E$, there exists $u_{h}\in\mathcal{U}$ such that $\varphi(k,h,u_{h})\in
Q$ for $k=1,\ldots,n$. We claim that $\mathcal{S}:=\{u_{h}\in\mathcal{U}%
;\ h\in E\}$ is an $(n,K,N_{\varepsilon}(Q))$-spanning set. In fact, since the
set $E$ is also $(n,\varepsilon)$-spanning, one finds for all $g\in K$ an
element $h\in E$ with
\[
\varrho(\varphi_{k,u_{h}}(g),\varphi_{k,u_{h}}(h))=\varrho(f_{0}^{k}%
(g),f_{0}^{k}(h))<\varepsilon
\]
for all $k\in\{0,\ldots,n\}$. This shows that $\varphi_{k,u_{h}}(g)\in
N_{\varepsilon}(Q)$ and hence proves the claim. It follows that
$r_{\mathrm{inv}}(n,K,N_{\varepsilon}(Q))\leq s_{n}(\varepsilon,K)$ implying
\[
h_{\mathrm{inv,out}}(K,N_{\varepsilon}(Q))=\limsup_{n\rightarrow\infty}%
\frac{1}{n}\log r_{\mathrm{inv}}(n,K,N_{\varepsilon}(Q))\leq\limsup
_{n\rightarrow\infty}\frac{1}{n}\log s_{n}(\varepsilon,K).
\]
Letting $\varepsilon$ tend to $0$, one finds by formula (\ref{Bow1}) for the
topological entropy
\[
h_{\mathrm{inv,out}}(K,Q)\leq h_{\mathrm{top}}(f_{0})=\sum_{|\lambda|>1}%
\log|\lambda|,
\]
where the sum is taken over all eigenvalues $\lambda$ of $(df_{0})_{e}$ with
$|\lambda|>1$.

(ii) It remains to show the reverse inequality. Recall that $(\pi
,\text{id}_{\mathcal{U}})$ is a semi-conjugacy between system
(\ref{lineardisc}) and the induced system (\ref{induced}) with $H=G^{-}$.
Hence Proposition \ref{proposition9} implies%
\[
h_{\mathrm{inv}}(K,N_{\varepsilon}(Q))\geq h_{\mathrm{inv}}(\pi
(K),N_{\varepsilon}(\pi(Q))),
\]
and it suffices to estimate the right hand side. By Kawan \cite[Lemma
A.3]{Kawa13}, there is $\varepsilon>0$ small enough such that $\overline
{N_{\varepsilon}(\pi(Q))}$ is compact. The $G$-invariant Borel measure $\mu$
on $G/G^{-}$ whose existence is guaranteed by Proposition \ref{borel}
satisfies $\mu(N_{\varepsilon}(\pi(Q)))\leq\mu(\overline{N_{\varepsilon}%
(\pi(Q))})<\infty$, because this measure is finite on compact sets. Consider
for $n\in\mathbb{N}$ an $(n,\pi(K),N_{\varepsilon}(\pi(Q)))$-spanning set
$\mathcal{S}=\{u_{1},\ldots,u_{r}\}$ with minimal cardinality $r=r_{\emph{inv}%
}(n,\pi(K),N_{\varepsilon}(\pi(Q)))$. Define for $j\in\{1,\ldots,r\},$
\[
K_{j}:=\{gG^{-}\in\pi(K);\ \bar{\varphi}(k,gG^{-},u_{j})\in N_{\varepsilon
}(\pi(Q)),\ \forall\ k=0,\ldots,n\}.
\]
Each $K_{j}$ is a Borel set and these sets cover $\pi(K)$ by the choice of
$\mathcal{S}$. Since for each $u\in U$ the map $\bar{f}_{u}$ and hence also
$\bar{\varphi}_{n,u_{j}}$ are diffeomorphisms on $G/G^{-}$, the set
$\bar{\varphi}_{n,u_{j}}(K_{j})$ is a Borel set. The inclusion $\bar{\varphi
}_{n,u_{j}}(K_{j})\subset N_{\varepsilon}(\pi(Q))$ implies
\[
\mu(\bar{\varphi}_{n,u_{j}}(K_{j}))\leq\mu(N_{\varepsilon}(\pi(Q)))<\infty.
\]
Using the left invariance of $\mu$ equality (\ref{linear2}) yields
\[
\mu(\bar{\varphi}_{n,u_{j}}(K_{j}))=\mu(\bar{\varphi}_{n,u_{j}}(e)\bar{f}%
_{0}^{n}(K_{j}))=\mu(\bar{f}_{0}^{n}(K_{j})).
\]
Since $\bar{f}_{0}$ is an diffeomorphism, $\bar{f}_{0}\circ\mathcal{L}%
_{g}=\mathcal{L}_{f_{0}(g)}\circ\bar{f}_{0}$ and $|\det d(\mathcal{L}%
_{g})_{hG^{-}}|=1$ for all $g,h\in G$, we obtain
\begin{align*}
|\det(d\bar{f}_{0}^{n})_{gG^{-}}|  &  =|\det d(\mathcal{L}_{f_{0}^{n}%
(g)})_{eG^{-}}||\det[(d\bar{f}_{0}^{n})_{eG^{-}}]||\det d(\mathcal{L}_{g^{-1}%
})_{gG^{-}}|\\
&  =|\det(d\bar{f}_{0}^{n})_{eG^{-}}|=\left\vert \det\left(  (df_{0})_{e}%
^{n}|_{\mathfrak{g}^{+,0}}\right)  \right\vert ,
\end{align*}
where the last equality follows from Proposition \ref{decomp}.

By the left invariance of $\mu$ we have
\begin{align*}
\mu(\bar{f}_{0}^{n}(K_{j}))  &  =\int_{\bar{f}_{0}^{n}(K_{j})}d\mu
(gG^{-})=\int_{K_{j}}|\det(d\bar{f}_{0}^{n})_{gG^{-}}|d\mu(gG^{-})\\
&  =\left\vert \det\left(  (df_{0})_{e}^{n}|_{\mathfrak{g}^{+,0}}\right)
\right\vert \int_{K_{j}}d\mu(gG^{-})=\left\vert \det\left(  (df_{0}%
)_{e}|_{\mathfrak{g}^{+,0}}\right)  \right\vert ^{n}\mu(K_{j})
\end{align*}
Together, these relations yield
\begin{align*}
\mu(\pi(K))  &  \leq\sum_{j=1}^{r}\mu(K_{j})\leq r\max_{1\leq j\leq r}%
\mu(K_{j})=r\max_{1\leq j\leq r}\dfrac{\mu(\bar{f}_{0}^{n}(K_{j}))}{\left\vert
\det\left(  (df_{0})_{e}|_{\mathfrak{g}^{+,0}}\right)  \right\vert ^{n}}\\
&  =r\max_{1\leq j\leq r}\dfrac{\mu(\bar{\varphi}_{n,u_{j}}(K_{j}%
))}{\left\vert \det\left(  (df_{0})_{e}|_{\mathfrak{g}^{+,0}}\right)
\right\vert ^{n}}\leq r\dfrac{\mu(N_{\varepsilon}(\pi(Q)))}{\left\vert
\det\left(  (df_{0})_{e}|_{\mathfrak{g}^{+,0}}\right)  \right\vert ^{n}},
\end{align*}
which implies that
\[
r_{\mathrm{inv}}(n,\pi(K),N_{\varepsilon}(\pi(Q)))=r\geq\dfrac{\mu(\pi
(K))}{\mu(N_{\varepsilon}(\pi(Q)))}\left\vert \det\left(  (df_{0}%
)_{e}|_{\mathfrak{g}^{+,0}}\right)  \right\vert ^{n}.
\]
Note that $\mu(\pi(K))>0$, because $\nu_{G}(K)>0$ by hypothesis (see Remark
\ref{remarklast}). Denoting $C:=\dfrac{\mu(\pi(K))}{\mu(N_{\varepsilon}%
(Q))}>0$ we get
\begin{align*}
h_{\mathrm{inv}}(\pi(K),N_{\varepsilon}(\pi(Q)))  &  \geq\limsup
_{n\rightarrow\infty}\dfrac{1}{n}\log\left(  C\left\vert \det\left(
(df_{0})_{e}|_{\mathfrak{g}^{+,0}}\right)  \right\vert ^{n}\right) \\
&  =\limsup_{n\rightarrow\infty}\dfrac{1}{n}\left(  \log C+n\log\left\vert
\det\left(  (df_{0})_{e}|_{\mathfrak{g}^{+,0}}\right)  \right\vert ^{n}\right)
\\
&  =\log\prod_{|\lambda|\geq1}|\lambda|=\sum_{|\lambda|>1}\log|\lambda|,
\end{align*}
where the product is taken over all eigenvalues $\lambda$ of $(df_{0})_{e}$
with $|\lambda|\geq1$ and the sum is taken over all eigenvalues $\lambda$ of
$(df_{0})_{e}$ with $|\lambda|>1$. Taking $\varepsilon\searrow0$ the result follows.
\end{proof}

\begin{remark}
The formula presented in Theorem \ref{Theo11} (ii) holds for the class of
solvable, connected and simply connected Lie groups $G$. In fact, by San
Martin \cite[Proposion 10.6]{SM16} all connected subgroups of $G$ are closed,
hence $G^{-}$ is closed. In Section \ref{Section5} we show that $G^{-}$ is
still closed even when $G$ is not solvable (see Theorem \ref{theo23}) and
Theorem \ref{Theo11} (ii) can be applied for discrete-time linear systems on a
broader class of Lie groups.
\end{remark}

\begin{example}
Let $(K,Q)$ be an admissible pair for the system defined in Example
\ref{example5}, where $Q$ is compact and $K$ has positive Haar measure. Note
that the matrix of the differential of $f_{0}$ is given by
\[
(df_{0})_{(1,0)}=\left[
\begin{array}
[c]{ccc}%
1 & 0 & \\
0 & e^{2} &
\end{array}
\right]  .
\]
In this case, $\mathfrak{g}^{-}=\{0\}$ and $G^{-}=\{(1,0)\}$, which is closed.
By Theorem \ref{Theo11} (ii), $h_{\mathrm{inv,out}}(K,Q)=2$.
\end{example}

\begin{example}
Again, consider an admissible pair $(K,Q)$ of the system presented in Example
\ref{Ex4}. Assume that $Q$ is compact and $K$ has positive Haar measure. The
matrix of $(df_{0})_{e}$ in the canonical basis is
\[
\left[
\begin{array}
[c]{ccc}%
1 & 1 & 0\\
0 & 1 & 0\\
0 & 1 & 1
\end{array}
\right]
\]
and we can see that the unique eigenvalue of $(df_{0})_{e}$ is $1$, hence
$h_{\mathrm{inv,out}}(K,Q)=0$ by Theorem \ref{Theo11}.
\end{example}

\begin{remark}
For linear control systems in Euclidean space, cf. Example \ref{Example2},
natural candidates for admissible pairs $(K,Q)$ such that $Q$ has compact
closure, as considered in Theorem \ref{Theo11}, can be obtained by Colonius,
Cossich and Santana \cite[Theorem 32]{Cocosa4} as follows: If $A$ is
hyperbolic, there exists a unique control set $D$ with nonvoid interior (i.e.,
a maximal set of approximate controllability with $\mathrm{int}D\not =%
\varnothing$) and it is bounded. Hence its closure $\overline{D}$ is compact
and for any compact subset $K\subset Q:=\overline{D}\,$\ the pair $(K,Q)$ is
admissible with compact $Q$. Results on control sets for continuous-time
linear systems on Lie groups are proved in Ayala, Da Silva, and \ Zsigmond
\cite{AyalDSZ17} and Ayala, Da Silva, Philippe, and Zsigmond \cite{AySilJuZs}.
For discrete-time linear systems on Lie groups, the control sets have not be studied.
\end{remark}

\section{Closedness of the stable subgroup\label{Section5}}

The formula for the outer invariance entropy in Theorem \ref{Theo11} has been
derived under the assumption that the stable subgroup $G^{-}$ is closed. In
this section we provide a sufficient condition for this property. The
arguments can also be applied to the unstable subgroup $G^{+}$. First we show
that these subgroups are simply connected.

\begin{proposition}
The stable subgroup $G^{-}$ and the unstable subgroup $G^{+}$ are simply connected.
\end{proposition}

\begin{proof}
We will just show this fact for $G^{-}$, because the proof for $G^{+}$ is
analogous. Denote the exponential map $\exp:\mathfrak{g}\rightarrow G$
restricted to $\mathfrak{g}^{-}$ by $\exp^{-}$, which is the exponential map
of $G^{-}$. Since $G^{-}$ is nilpotent and connected, $\exp^{-}:\mathfrak{g}%
^{-}\rightarrow G^{-}$ is a covering map, hence it is surjective, continuous
and open. Hence, if we show that $\exp^{-}$ is injective, it will be a
homeomorphism. Therefore simple connectedness of $\mathfrak{g}^{-}$ implies
simple connectedness of $G^{-}$.

In order to show that $\exp^{-}$ is injective, let $X,Y\in\mathfrak{g}^{-}$
such that $\exp^{-}X=\exp^{-}Y$. Consider open neighborhoods $V$ and $U$ of
$0\in\mathfrak{g}$ and $e\in G$, respectively, such that $\exp:V\rightarrow U$
is a diffeomorphism. Remark \ref{Obs12} implies that for all $Z\in
\mathfrak{g}^{-}$ it holds
\[
|(df_{0})_{e}^{n}(Z)|\leq c\mu^{n}|Z|\text{ for all }n\in\mathbb{N},
\]
for some $c\geq1$ and $\mu\in(0,1)$. Consider $n$ large enough such that
$(df_{0})_{e}^{n}(Y),\,\allowbreak(df_{0})_{e}^{n}(Y)\in V$. Hence,
\[
\exp^{-}((df_{0})_{e}^{n}(X))=f_{0}^{n}(\exp^{-}X)=f_{0}^{n}(\exp^{-}%
Y)=\exp^{-}((df_{0})_{e}^{n}(Y)).
\]
The injectivity of $\exp|_{V}$ implies that $\exp^{-}|_{V}=\exp|_{V\cap
\mathfrak{g}^{-}}$ is injective, and hence $(df_{0})_{e}^{n}(X)=(df_{0}%
)_{e}^{n}(Y)$. Since $(df_{0})_{e}|_{\mathfrak{g}^{-}}$ is an isomorphism, we
obtain $X=Y$.
\end{proof}

\begin{remark}
The maps $\exp^{\pm}=\exp|_{G^{\pm}}:\mathfrak{g}^{\pm}\rightarrow G^{\pm}$
are diffeomorphisms, because $G^{\pm}$ are connected, simply connected and
nilpotent (see Knapp \cite[Theorem 1.127]{Knapp} or San Martin \cite[Theorem
10.8]{SM16}).
\end{remark}

The following theorem presents the announced sufficient condition for the
closedness of $G^{-}$ (and $G^{+}$).

\begin{theorem}
\label{theo23} If $G$ is simply connected, the stable and the unstable
subgroup $G^{-}$ and $G^{+}$, resp., of the uncontrolled system are closed.
\end{theorem}

\begin{proof}
We will just show this fact for $G^{-}$, because the proof for $G^{+}$ is
analogous. Since $G$ is simply connected and $G^{-}$ is connected, $G^{-}$
cannot be dense in $G$, hence the open set $G\setminus\overline{G^{-}}$ is
non-empty. The continuity of $\exp:\mathfrak{g}\rightarrow G$ implies that
$\exp^{-1}\left(  G\setminus\overline{G^{-}}\right)  $ is a non-empty open set
of $\mathfrak{g}$. In order to show that $\mathfrak{g}^{+,0}\setminus
\{0\}\subset\exp^{-1}\left(  G\setminus\overline{G^{-}}\right)  $, consider
$Z\in\mathfrak{g}^{+,0}\setminus\{0\}$ and suppose that $\exp(Z)\in
\overline{G^{-}}$. Then there is a sequence $(g_{n})\subset G^{-}$ such that
$g_{n}\rightarrow\exp(Z)$. Since $\exp^{-}:=\exp|_{\mathfrak{g}^{-}%
}:\mathfrak{g}^{-}\rightarrow G^{-}$ is surjective, there is a sequence
$(X_{n})\subset\mathfrak{g}^{-}$ with $g_{n}=\exp(X_{n})$. Hence $\exp
(X_{n})\rightarrow\exp(Z)$, so $X_{n}\rightarrow Z$. The closedness of
$\mathfrak{g}^{-}$ implies that $Z\in\mathfrak{g}^{-}\cap\mathfrak{g}^{+,0}$,
that is, $Z=0$ which is a contradiction.

Assume that $G^{-}$ is not closed and denote by $\overline{\mathfrak{g}^{-}}$
the Lie algebra of $\overline{G^{-}}$. Then $\mathfrak{g}^{-}$ is properly
contained in $\overline{\mathfrak{g}^{-}}$. Consider $X\in\overline
{\mathfrak{g}^{-}}\setminus\mathfrak{g}^{-}$. Then there are $Y\in
\mathfrak{g}^{+}$, $Z\in\mathfrak{g}^{0}$ and $X^{-}\in\mathfrak{g}^{-}$ such
that
\[
X=Y+Z+X^{-},
\]
where $Y$ or $Z$ is non-zero. Hence $Y+Z=X-X^{-}\in\mathfrak{g}^{+,0}%
\cap\overline{\mathfrak{g}^{-}}$. But it implies that $\exp^{-1}\left(
G\setminus\overline{G^{-}}\right)  \cap\overline{\mathfrak{g}^{-}}%
\neq\emptyset$, which is a contradiction, because $\exp^{-1}\left(
G\setminus\overline{G^{-}}\right)  $ and $\overline{\mathfrak{g}^{-}}$ are disjoint.
\end{proof}

The following example shows that the subgroups $G^{+}$ and $G^{-}$ may not be
closed, if the group $G$ is not simply connected.

\begin{example}
Consider the following automorphism on the torus $\mathbb{T}^{2}%
=\mathbb{R}^{2}/\mathbb{Z}^{2}$ which, naturally, is not simply connected:
\[
f_{0}((x,y)+\mathbb{Z}^{2})=(2x+y,x+y)+\mathbb{Z}^{2}.
\]
Then $(df_{0})_{e}:\mathbb{R}^{2}\rightarrow\mathbb{R}^{2}$ is given by
\[
(df_{0})_{e}=\left[
\begin{array}
[c]{cc}%
2 & 1\\
1 & 1
\end{array}
\right]  .
\]
Hence, the eigenvalues of $(df_{0})_{e}$ are $\lambda_{1}=\frac{1}{2}%
(3+\sqrt{5})>1$ and $\lambda_{2}=\frac{1}{2}(3-\sqrt{5})\in(0,1)$. Moreover,
we have
\[
\mathfrak{g}^{+}=\text{span}\left\{  \left(  \frac{1}{2}(3+\sqrt{5}),1\right)
\right\}  ,\quad\mathfrak{g}^{-}=\text{span}\left\{  \left(  \frac{1}%
{2}(3-\sqrt{5}),1\right)  \right\}  .
\]
Therefore
\begin{align*}
G^{+}  &  =\left\{  \exp\left(  t\left(  \frac{1}{2}(3+\sqrt{5}),1\right)
\right)  ;\ t\in\mathbb{R}\right\}  ,\\
G^{-}  &  =\left\{  \exp\left(  t\left(  \frac{1}{2}(3-\sqrt{5}),1\right)
\right)  ;\ t\in\mathbb{R}\right\}  .
\end{align*}
Both $G^{+}$ and $G^{-}$ are irrational flows on $\mathbb{T}^{2}$ which are
dense in $\mathbb{T}^{2}$, hence they are not closed.
\end{example}

\textbf{Acknowledgement.} We are grateful to an anonymous reviewer whose
comments were very helpful for the revision of the paper.


\begin{thebibliography}{99}                                                                                               %


\bibitem {AyaDS16}V. Ayala and A. Da Silva, Dynamics of endomorphisms of Lie
groups and their topological entropy, (to appear)

\bibitem {AySilJuZs}V. Ayala, A. Da Silva, J. Philippe, and G. Zsigmond,
Control sets of linear systems on semi-simple Lie groups. J. Diff. Equ. 269(1)
(2020), pp. 449-466.

\bibitem {AyalDSZ17}V. Ayala, A. Da Silva, and G. Zsigmond, Control sets of
linear systems on Lie groups, Nonlinear Differential Equations and
Applications 24(8) (2017), pp. 7-15.

\bibitem {Bow}R. Bowen. Entropy for group endomorphisms and homogeneous
spaces, Trans. Amer. Math. Soc. 153 (1971), pp. 401-414.

\bibitem {Cocosa4}F. Colonius, J.A.N. Cossich and A. Santana, Controllability
properties and invariance pressure for linear discrete-time systems, J.
Dynamics and Differential Equations (2021), DOI 10.1007/s10884-021-09966-4.

\bibitem {ColoniusKawan}F. Colonius and C. Kawan, Invariance entropy for
control systems. SIAM J. Control Optim. 48 (3) (2009), pp. 1701-1721.

\bibitem {DaSil14}A. Da Silva, Outer invariance entropy for linear systems on
Lie groups, SIAM J. Control Optim. 52(6) (2014), pp. 3917-3934.

\bibitem {DaSil16}A. Da Silva, Controllability of linear systems on solvable
Lie groups, SIAM J. Control Optim. 54(1) (2016), pp. 372-390.

\bibitem {DaSilK16}A. Da Silva and C. Kawan, Invariance entropy of hyperbolic
control sets, Discrete and Continuous Dynamical Systems 36 (2016), pp. 97--136.

\bibitem {DaSilK19}A. Da Silva and C. Kawan, Lyapunov exponents and partial
hyperbolicity of chain control sets on flag manifolds, Israel J. Math. 232
(2019), pp. 947--1000.

\bibitem {Knapp}A. Knapp, Lie Groups Beyond an Introduction, Second Edition,
Birkh\"{a}user, 2002.

\bibitem {Kawa11b}C. Kawan, Invariance entropy of control sets, SIAM J.
Control Optim. 49 (2011), pp. 732-751.

\bibitem {Kawa13}C. Kawan, Invariance Entropy for Deterministic Control
Systems. An Introduction. LNM Vol. 2089, Springer, Berlin, 2013.

\bibitem {NEMM04}G. Nair, R. J. Evans, I. Mareels, and W. Moran, Topological
feedback entropy and nonlinear stabilization, IEEE Trans. Aut. Control 49
(2004), pp. 1585--1597.

\bibitem {SM16}L. San Martin, Grupos de Lie, Editora Unicamp, 2016 Campinas.

\bibitem {Son98}E. Sontag, Mathematical Control Theory. Deterministic Finite
Dimensional Systems, 2nd ed. Springer-Verlag, New York 1998.

\bibitem {Walt82}P. Walters, An Introduction to Ergodic Theory,
Springer-Verlag, 1982.
\end{thebibliography}
\end{document}